\documentclass[a4paper,leqno]{amsart}
\usepackage{amsmath, amssymb, amsthm}

\usepackage[mathscr]{eucal}

\newtheorem{thm}{\textbf{Theorem}}[section]
\newtheorem{prop}[thm]{\textbf{Proposition}}
\newtheorem{cor}[thm]{\textbf{Corollary}}

\theoremstyle{definition}
\newtheorem{defn}[thm]{{\rm Definition}}

\newcommand{\olapla}{\overline\bigtriangleup}
\newcommand{\onabla}{\overline\nabla}

\newcommand{\p}{\phi}

\title{k-harmonic curves into a Riemannian manifold with constant sectional curvature}

\author{Shun Maeta}

\curraddr{Nakakuki 3-10-9 Oyama-shi Tochigi
Japan}
\email{shun.maeta@gmail.com}


\begin{document}




\begin{abstract}
In \cite{jell1}, J.Eells and L. Lemaire introduced $k$-harmonic maps,
 and T. Ichiyama, J. Inoguchi and H.Urakawa \cite{tijihu1} showed the first variation formula.
 In this paper,we describe the ordinary differential equations of $3$-harmonic curves into a Riemannian manifold with 
constant sectional curvature, 
 and show biharmonic curve is k-harmonic curve $(k\geq 2)$. 
\end{abstract}

\maketitle


\vspace{10pt}
\begin{flushleft}
{\large {\bf Introduction}}
\end{flushleft}
Theory of harmonic maps has been applied into various fields in differential geometry.
 The harmonic maps between two Riemannian manifolds are
 critical maps of the energy functional $E(\p)=\frac{1}{2}\int_M\|d\p\|^2v_g$, for smooth maps $\p:M\rightarrow N$.
 
On the other hand, in 1981, J. Eells and L. Lemaire \cite{jell1} proposed the problem to consider the {\em $k$-harmonic maps}:
 they are critical maps of the functional 
 \begin{align*}
 E_{k}(\p)=\int_Me_k(\p)v_g,\ \ (k=1,2,\dotsm),
 \end{align*}
 where $e_k(\p)=\frac{1}{2}\|(d+d^*)^k\p\|^2$ for smooth maps $\p:M\rightarrow N$.
G.Y. Jiang \cite{jg1} studied the first and second variation formulas of the bi-energy $E_2$, 
and critical maps of $E_2$ are called {\em biharmonic maps}. There have been extensive studies on biharmonic maps.
 
 Recently, in 2009, T. Ichiyama, J. Inoguchi and H. Urakawa \cite{tijihu1} studied the first variation formula of the 
$k$-energy $E_k$,
 whose critical maps are called $k$-harmonic maps.
 Harmonic maps are always $k$-harmonic maps by definition.
 In this paper, we study $k$-harmonic curves. 
 
In $\S \ref{preliminaries}$, we introduce notation and fundamental formulas of the tension field. 

In $\S \ref{k-harmonic}$, we show biharmonic curves into a Riemannian manifold with constant sectional curvature 
is always $k$-harmonic curves.


\section{Preliminaries}\label{preliminaries}
Let $(M,g)$ be an $m$ dimensional Riemannian manifold,
 $(N,h)$ an $n$ dimensional one,
 and $\p:M\rightarrow N$, a smooth map.
 We use the following notation.
 The second fundamental form $B(\p)$
 of $\p$ is a covariant differentiation $\widetilde\nabla d\p$ of $1$-form $d\p$,
 which is a section of $\odot ^2T^*M\otimes \p^{-1}TN$.
For every $X,Y\in \Gamma (TM)$, let
 \begin{equation}
 \begin{split}
 B(X,Y)
&=(\widetilde\nabla d\p)(X,Y)=(\widetilde\nabla_X d\p)(Y)\\
&=\overline\nabla_Xd\p(Y)-d\p(\nabla_X Y)=\nabla^N_{d\p(X)}d\p(Y)-d\p(\nabla_XY). 
 \end{split}
 \end{equation}
 Here, $\nabla, \nabla^N, \overline \nabla, \widetilde \nabla$ are the induced connections on the bundles $TM$,
 $TN$, $\p^{-1}TN$ and $T^*M\otimes \p^{-1}TN$, respectively.
 
 If $M$ is compact,
 we consider critical maps of the energy functional
 \begin{align}
 E(\p)=\int_M e(\p) v_g,
 \end{align}
where $e(\p)=\frac{1}{2}\|d\p\|^2=\sum^m_{i=1}\frac{1}{2}\langle d\p(e_i),d\p(e_i)\rangle$
 which is called the {\em enegy density} of $\p$, and the inner product 
 $\langle \cdot ,\cdot \rangle$ is a Riemannian metric $h$. 
 The {\em tension \ field} $\tau(\p)$ of $\p$ is defined by
 \begin{align}
 \tau(\p)=\sum^{m}_{i=1}(\widetilde \nabla d\p)(e_i,e_i)=\sum^m_{i=1}(\widetilde \nabla _{e_i}d\p)(e_i).
 \end{align}
 Then, $\p$ is a {\em harmonic map} if $\tau(\p)=0$.
 
 
 And we define
\begin{align}
\olapla
=\onabla^* \onabla
=-\sum^m_{k=1}(\onabla_{e_k}\onabla_{e_k}
-\onabla_{\nabla_{e_k}e_k}), 
\end{align}
 is the {\em rough Laplacian}.

 And we define $\mathscr{R}$ as follows :
\begin{align}
\mathscr{R}(V):=\sum^m_{i=1}R^N(V,d\p(e_i))d\p(e_i),\ \ \ V\in \Gamma(\p^{-1}TN),
\end{align}
where, 
\begin{align*}
R^N(U,V)=\nabla^N_{U}\nabla^N_{V}-\nabla^N_{V}\nabla^N_{U}-\nabla^N_{[U,V]},\ \ \ \ \ U,V\in \Gamma(TN),
\end{align*}
 is the curvature tensor of $(N,h)$.

\section{$k$-harmonic curves into a Riemannian manifold with constant sectional curvature}\label{k-harmonic}

In this section, we consider curves into a Riemannian manifold with constant sectional curvature. 
Then, we show the necessary and sufficient condition of 3-harmonic curve,
 and biharmonic curve is k-harmonic curve. 

\vspace{5pt}

 J. Eells and L. Lemaire \cite{jell1} proposed the notation of $k$-harmonic maps. 
The Euler-Lagrange equations for the $k$-harmonic maps was shown by
T. Ichiyama, J. Inoguchi and H. Urakawa \cite{tijihu1}.
We first recall it briefly.

\begin{thm}[\cite{tijihu1}]\label{kharmonic}
Let $k=2,3,\dotsm.$ Then, we have
\begin{align*}
\left.\frac{d}{dt}\right|_{t=0}E_k(\phi_t)
=-\int_M\langle \tau_k(\phi),V \rangle v_g,
\end{align*}
where
\begin{align*}
\tau _k(\phi)
:=J\left(\overline \triangle ^{(k-2)}\tau (\phi )\right)
=\overline \triangle \left (\overline \triangle ^{(k-2)}\tau (\phi )\right)
-\mathscr{R} \left( \overline \triangle ^{(k-2)}\tau (\phi ) \right),
\end{align*}
and
\begin{align*}
\overline \triangle ^{(k-2)}\tau (\phi )
=\underbrace{\overline \triangle \dotsm \overline \triangle}_{k-2} 
\tau(\phi).
\end{align*}
\end{thm}

\vspace{10pt}
 As a corollary of this theorem, we have
\begin{cor}[\cite{tijihu1}]
$\phi :(M,g)\rightarrow (N,h)$ is a $k$-harmonic map if
\begin{equation}\label{propkharmonic}
\begin{split}
\tau _k(\phi)
:=J\left(\overline \triangle ^{(k-2)}\tau (\phi )\right)
=\overline \triangle \left (\overline \triangle ^{(k-2)}\tau (\phi )\right)
-\mathscr{R} \left( \overline \triangle ^{(k-2)}\tau (\phi ) \right)=0.
\end{split}
\end{equation}
\end{cor}

\vspace{10pt}

We say for a $k$-harmonic map to be {\em proper} if it is not harmonic. 

\pagebreak

Let us recall the definition of the Frenet frame.

\begin{defn}
The Frenet frame $\{e_i\}_{i=1,\dotsm n}$ associated to a curve 
$\gamma :I \in \mathbb{R}\rightarrow (N^n, \langle \cdot, \cdot \rangle)$,
 parametrized by arc length, is the orthonormalisation 
of the (n+1)-uple 
$\{\nabla^{N(k)}_{d\gamma(\frac{\partial}{\partial t})}d\gamma (\frac{\partial}{\partial t})\}_{k=1,\dotsm, n}$,
 described by
\begin{align*}
&e_1=d\gamma\left(\frac{\partial}{\partial t}\right),\\
&\nabla^N_{d\gamma(\frac{\partial}{\partial t})}e_1=\kappa_1 e_2,\\
&\nabla^N_{d\gamma(\frac{\partial}{\partial t})}e_i=-\kappa_{i-1}e_{i-1}+\kappa_i e_{i+1}\ \ (i=2,\dotsm, n-1),\\
&\nabla^N_{d\gamma(\frac{\partial}{\partial t})}e_n=-\kappa_{n-1} e_{n-1},\\
\end{align*}
where the functions ${\kappa_1, \kappa_2, \dotsm, \kappa_{n-1}}$ are called the curvatures of $\gamma$.
 Note that $e_1=\gamma'$
 is the unit tangent vector field along the curve. 
\end{defn}

\vspace{10pt}
First, we show the necessary and sufficient condition of k-harmonic curves 
into a Riemannian manifold with constant sectional curvature.

\begin{prop}\label{k-harmonic c s}
Let $\gamma:I\rightarrow (N^n,\langle \cdot, \cdot \rangle)$ be a smooth curve parametrized by arc length 
 from an open interval of $\mathbb{R}$ into a Riemannian manifold $(N^n,\langle \cdot, \cdot \rangle)$ with constant sectional curvature $K$.
 Then, $\gamma $ is $k$-harmonic if and only if,
\begin{align}
(\nabla^N_{\gamma'}\nabla^N_{\gamma'})^{k-1}\tau(\gamma)
-K\{(\nabla^N_{\gamma'}\nabla^N_{\gamma'})^{k-2}\tau(\gamma)
-\langle\gamma',(\nabla^N_{\gamma'}\nabla^N_{\gamma'})^{k-2}\tau(\gamma)\rangle\gamma'\}=0.
\end{align}
\end{prop}

\vspace{10pt}

\begin{proof}
\begin{align*}
\olapla\tau(\gamma)&=(-1)^{k-1}(\nabla^N_{\gamma'}\nabla^N_{\gamma'})^{k-1}\tau(\gamma),\\
\mathscr{R}(\olapla^{k-2}\tau(\gamma))
&=K\{(-1)^{k-2}(\nabla^N_{\gamma'}\nabla^N_{\gamma'})^{k-2}\tau(\gamma)
-\langle\gamma',(-1)^{k-2}(\nabla^N_{\gamma'}\nabla^N_{\gamma'})^{k-2}\tau(\gamma)\rangle\gamma'\}.
\end{align*}
Therefore, we have Proposition $\ref{k-harmonic c s}$.
\end{proof}

\vspace{10pt}

Using Proposition $\ref{k-harmonic c s}$, we show the necessary and sufficient condition of biharmonic curve
 and $3$-harmonic curve, respectively.

\begin{prop}\label{biharmonic c s}
Let $\gamma:I\rightarrow (N^n,\langle \cdot, \cdot \rangle)$ be a smooth curve parametrized by arc length 
 from an open interval of $\mathbb{R}$ into a Riemannian manifold $(N^n,\langle \cdot, \cdot \rangle)$ with constant sectional curvature $K$.
 Then, $\gamma $ is proper biharmonic if and only if,
\begin{align}
\begin{cases}
\kappa_1^2+\kappa_2^2=K,\\
\kappa_1=constant\neq 0,\\
\kappa_2=constant,\\
\kappa_2\kappa_3=0.
\end{cases}
\end{align}
\end{prop}

\vspace{3pt}

\begin{proof}
$\tau(\gamma)=\kappa_1e_2$.
So we caluculate $(\nabla_{\gamma'}^N\nabla_{\gamma'}^N)(\kappa_1e_2)$ as follows.
\begin{align}
(&\nabla_{\gamma'}^N\nabla_{\gamma'}^N)(\kappa_1e_2) \label{gamma2}\\
&=-3\kappa_1\kappa_1'e_1+(\kappa_1''-\kappa_1^3-\kappa_1\kappa_2^2)e_2 \notag 
+(2\kappa_1'\kappa_2+\kappa_1\kappa_2')e_3+\kappa_1\kappa_2\kappa_3e_4.
\end{align}
Using Proposition $\ref{k-harmonic c s}$, and $\kappa_1\neq 0$, we have Proposition $\ref{biharmonic c s}.$
\end{proof}


\begin{prop}\label{3-harmonic c s}
Let $\gamma:I\rightarrow (N^n,\langle \cdot, \cdot \rangle)$ be a smooth curve parametrized by arc length 
 from an open interval of $\mathbb{R}$ into a Riemannian manifold $(N^n,\langle \cdot, \cdot \rangle)$ with constant sectional curvature $K$.
 Then, $\gamma $ is $3$-harmonic if and only if,
\begin{align}
\begin{cases}
-2\kappa'_1\kappa''_1-\kappa_1\kappa_1^{(3)}+2\kappa^3_1\kappa_1'+\kappa_1\kappa_1'\kappa_2^2+\kappa_1^2\kappa_2\kappa_2'=0\\
-15\kappa_1(\kappa_1')^2- 10\kappa_1^2\kappa_1''+ \kappa_1^5+ 2\kappa_1^3\kappa_2^2+\kappa_1^{(4)}
-6\kappa_1''\kappa_2^2 -12\kappa_1'\kappa_2\kappa_2'\\
-3\kappa_1(\kappa_2')^2 -4\kappa_1\kappa_2\kappa_2''
+\kappa_1\kappa_2^4 +\kappa_1\kappa_2^2\kappa_3^2
+K\{\kappa_1''-\kappa_1^3-\kappa_1\kappa_2^2\}=0\\
4\kappa_1^{(3)}\kappa_2
-9\kappa_1^2\kappa_1'\kappa_2 -4\kappa_1'\kappa_2^3 -6\kappa_1\kappa_2^2\kappa_2' +6\kappa_1''\kappa_2'-\kappa_1^3\kappa_2'\\
+4\kappa_1'\kappa_2''+\kappa_1\kappa_2^{(3)}-4\kappa_1'\kappa_2\kappa_3^2
-3\kappa_1\kappa_2'\kappa_3^2-3\kappa_1\kappa_2\kappa_3\kappa_3'
+K\{2\kappa_1'\kappa_2+\kappa_1\kappa_2'\}=0\\
6\kappa_1''\kappa_2\kappa_3 -\kappa_1^3\kappa_2\kappa_3 -\kappa_1\kappa_2^3\kappa_3 
+8\kappa_1'\kappa_2'\kappa_3 +3\kappa_1\kappa_2''\kappa_3
-\kappa_1\kappa_2\kappa_3^3\\
+4\kappa_1'\kappa_2\kappa_3' +3\kappa_1\kappa_2'\kappa_3' +\kappa_1\kappa_2\kappa_3'' -\kappa_1\kappa_2\kappa_3\kappa_4^2
+K\{\kappa_1\kappa_2\kappa_3\}=0\\
4\kappa_1'\kappa_2\kappa_3\kappa_4 +3\kappa_1\kappa_2'\kappa_3\kappa_4
+2\kappa_1\kappa_2\kappa_3'\kappa_4+\kappa_1\kappa_2\kappa_3\kappa_4'=0\\
\kappa_1\kappa_2\kappa_3\kappa_4\kappa_5=0
\end{cases}
\end{align}
\end{prop}

\vspace{5pt}

\begin{proof}
We caluculate $(\nabla_{\gamma'}^N\nabla_{\gamma'}^N)^2\tau(\gamma)$ as follows.
\begin{align*}
(\nabla_{\gamma'}^N&\nabla_{\gamma'}^N)^2\tau(\gamma)\\
&\hspace{-5pt}=(-10\kappa'_1\kappa''_1-5\kappa_1\kappa_1^{(3)}+10\kappa^3_1\kappa_1'+5\kappa_1\kappa_1'\kappa_2^2+5\kappa_1^2\kappa_2\kappa_2')e_1\\
&+(-15\kappa_1(\kappa_1')^2 -10\kappa_1^2\kappa_1''+\kappa_1^5 +2\kappa_1^3\kappa_2^2+\kappa_1^{(4)}\\
&\ \ -6\kappa_1''\kappa_2^2 -12\kappa_1'\kappa_2\kappa_2' -3\kappa_1(\kappa_2')^2 -4\kappa_1\kappa_2\kappa_2''
+\kappa_1\kappa_2^4 +\kappa_1\kappa_2^2\kappa_3^2)e_2\\
&+(4\kappa_1^{(3)}\kappa_2-9\kappa_1^2\kappa_1'\kappa_2
-4\kappa_1'\kappa_2^3-6\kappa_1\kappa_2^2\kappa_2'\\
&\ \ +6\kappa_1''\kappa_2'-\kappa_1^3\kappa_2'
+4\kappa_1'\kappa_2''+\kappa_1\kappa_2^{(3)}-4\kappa_1'\kappa_2\kappa_3^2\\
&\ \ -3\kappa_1\kappa_2'\kappa_3^2-3\kappa_1\kappa_2\kappa_3\kappa_3')e_3\\
&+(6\kappa_1''\kappa_2\kappa_3-\kappa_1^3\kappa_2\kappa_3-\kappa_1\kappa_2^3\kappa_3 +8\kappa_1'\kappa_2'\kappa_3\\
&\ \ +3\kappa_1\kappa_2''\kappa_3-\kappa_1\kappa_2\kappa_3^3+4\kappa_1'\kappa_2\kappa_3'\\
&\ \ +3\kappa_1\kappa_2'\kappa_3'+\kappa_1\kappa_2\kappa_3''-\kappa_1\kappa_2\kappa_3\kappa_4^2)e_4\\
&+(4\kappa_1'\kappa_2\kappa_3\kappa_4 +3\kappa_1\kappa_2'\kappa_3\kappa_4
+2\kappa_1\kappa_2\kappa_3'\kappa_4+\kappa_1\kappa_2\kappa_3\kappa_4')e_5\\
&+\kappa_1\kappa_2\kappa_3\kappa_4\kappa_5e_6
\end{align*}
 Using Proposition $\ref{k-harmonic c s}$, and $(\ref{gamma2})$,
we have Proposition $\ref{3-harmonic c s}$. 
\end{proof}

\vspace{10pt}

We showed biharmonic curve is $k$-harmonic curve into 2-dimensional unit sphere \cite{sm1}.
 We generalize this as following.
 
\begin{thm}\label{2-harmonic is k-harmonic}
Let $\gamma:I\rightarrow (N^n,\langle \cdot, \cdot \rangle)$ be a smooth curve parametrized by arc length 
 from an open interval of $\mathbb{R}$ into a Riemannian manifold $(N^n,\langle \cdot, \cdot \rangle)$ with constant sectional curvature $K$.
 Then, biharmonic is $k$-harmonic $(k\geq 2)$.
\end{thm}

\begin{proof}
By Proposition $\ref{biharmonic c s}$, $\gamma$ is proper biharmonic if and only if
\begin{align*}
\begin{cases}
\kappa_1^2+\kappa_2^2=K,\\
\kappa_1=constant\neq 0,\\
\kappa_2=constant,\\
\kappa_2\kappa_3=0.
\end{cases}
\end{align*}
 Then, we caluculate $(\nabla_{\gamma'}^N\nabla_{\gamma'}^N)^k\tau(\gamma)$.
\begin{align*}
(\nabla_{\gamma'}^N\nabla_{\gamma'}^N)^k\tau(\gamma)
&=(-1)^k\kappa_1(\kappa_1^2+\kappa_2^2)^ke_2\\
&=(-1)^k\kappa_1K^ke_2.
\end{align*}
So, we have
\begin{align*}
&\hspace{-5pt}(\nabla_{\gamma'}^N\nabla_{\gamma'}^N)^{k-1}\tau(\gamma)
+K\{(\nabla_{\gamma'}^N\nabla_{\gamma'}^N)^{k-2}\tau(\gamma)-\langle \gamma', (\nabla_{\gamma'}^N\nabla_{\gamma'}^N)^{k-2}\tau(\gamma)\rangle\gamma'\}\\
&=(-1)^{k-1}\kappa_1K^{k-1}e_2+K(-1)^{k-2}\kappa_1K^{k-2}e_2\\
&=0.
\end{align*}
And harmonic is always $k$-harmonic.
 So we have Theorem $\ref{2-harmonic is k-harmonic}$.

\end{proof}



\end{document}